\documentclass[reqno,a4paper]{amsart}
\usepackage{amsmath,amsthm,amssymb,tabularx}

\theoremstyle{plain}
\newtheorem{thm}{Theorem}[section]
\newtheorem{lm}[thm]{Lemma}
\newtheorem{cor}[thm]{Corollary}
\newtheorem{prop}[thm]{Proposition}

\theoremstyle{definition}
\newtheorem{definition}[thm]{Definition}
\newtheorem{defn}[thm]{Definition}

\newtheorem{ex}[thm]{Example}

\newtheorem{rem}[thm]{Remark}
\newtheorem{nots}[thm]{Notations}
\allowdisplaybreaks

\newcommand{\beq}{\begin{equation}}
\newcommand{\eeq}{\end{equation}}
\newcommand{\bga}{\begin{gather*}}
\newcommand{\ega}{\end{gather*}}
\newcommand{\bal}{\begin{align*}}
\newcommand{\eal}{\end{align*}}
\newcommand{\bit}{\begin{itemize}}
\newcommand{\eit}{\end{itemize}}
\newcommand{\btm}{\begin{tm}}
\newcommand{\etm}{\end{tm}}
\newcommand{\blm}{\begin{lm}}
\newcommand{\elm}{\end{lm}}
\newcommand{\bcor}{\begin{cor}}
\newcommand{\ecor}{\end{cor}}
\newcommand{\bprop}{\begin{prop}}
\newcommand{\eprop}{\end{prop}}
\newcommand{\bex}{\begin{ex}}
\newcommand{\eex}{\end{ex}}
\newcommand{\bpr}{\begin{proof}}
\newcommand{\epr}{\end{proof}}
\newcommand{\brem}{\begin{rem}}
\newcommand{\erem}{\end{rem}}
\newcommand{\bdf}{\begin{definition}}
\newcommand{\edf}{\end{definition}}
\newcommand{\bnots}{\begin{nots}}
\newcommand{\enots}{\end{nots}}

\def\C{\mathbb{C}}
\def\N{\mathbb{N}}
\def\R{\mathbb{R}}

\def\Id{\mathbb I}
\def\D{{\mathbb D}}

\def\e{\varepsilon}
\let\a\alpha
\let\b\beta
\let\g\gamma

\let\cal\mathcal

\def \le {\leqslant}
\def \ge {\geqslant}
\let \<\langle
\let \>\rangle
\let\phi\varphi
\let\kappa\varkappa

\let\cal\mathcal

\def\otwo{\underset2\otimes}

\def\G{\mathbb{G}}

\begin{document}
\author{Yulia Kuznetsova}
\title[Invariant means on quantum WAP functionals]{Invariant means on subspaces of quantum weakly almost periodic functionals}
\address{University of Bourgogne Franche-Comt\'e, 16 route de Gray, 25030 Besan\c con, France}
\email{yulia.kuznetsova@univ-fcomte.fr}

\begin{abstract}
Let $M$ be a Hopf--von Neuman algebra with the predual $M_*$ and $WAP(M)$ the subspace in $M$ composed of weakly almost periodic functionals on $M_*$. The main example of such an algebra is $M=L^\infty(\mathbb G)$ for a locally compact quantum group $\mathbb G$. We define a pair of left/right spaces $WAP_{iso,l}(M)$ and $WAP_{iso,r}(M)$ inside $WAP(M)$ and prove that they carry invariant means. These spaces are currently the widest known to admit invariant means in the quantum setting. In the case when $M=L^\infty(G)$ and $G$ is a locally compact group, these spaces are equal to $WAP(G)$.
\end{abstract}

\maketitle


\section{Introduction}

The notion of a weakly almost periodic function is stated naturally in the context of semitopological semigroups, though it appeared first in the group case. A semigroup $S$ is called semitopological if it is endowed with a topology such that left and right multiplication maps $l_s:t\mapsto st$, $r_s: t\mapsto ts$ are continuous on $S$ for every $s\in S$. This condition is strictly weaker than the joint continuity of multiplication. Let $C_b(S)$ be the space of continuous bounded functions on $S$. The semigroup acts on it by translations: $L_s(f)=f\circ l_s$, $R_s(f)=f\circ r_s$ for $f\in C_b(S)$, $s\in S$. A function $f\in C_b(S)$ is {\sl weakly almost periodic} if its orbit $Orb_S(f)=\{ L_s(f): s\in S\}$ is relatively weakly compact in $C_b(S)$. Considering the right orbit instead leads to the same class of functions.

In the case when $S=G$ is a locally compact group, this allows also a linearized formulation. In fact \cite{kitchen,dunkl-ramirez}, a function $f$ is weakly almost periodic exactly when the map $\mu\mapsto \mu*f$, equivalently $\mu\mapsto f*\mu$, is weakly compact from $L^1(G)$ to $C_b(G)$. This motivated the following, more general definition \cite{lau-loy}. If $A$ is a Banach algebra and $A^*$ its dual, then $A^*$ is a module over $A$: for $\phi\in A^*$ and $a\in A$, we define $a.\phi\in A^*$ by $(a.\phi)(b)=\phi(ba)$, and respectively $(\phi.a)(b) = \phi(ab)$, $b\in A$. Now $\phi$ is weakly almost periodic if the map $a\mapsto a.\phi$, $A\to A^*$, is weakly compact (the image of the unit ball of $A$ is relatively weakly compact). In this paper, we denote by $WAP(A^*)$ the space of weakly almost periodic functionals on $A$ (and not $WAP(A)$, because the focus will be on the dual). 

Our context is between the two. The main object is a Hopf--von Neumann algebra $M$, that is, a von Neumann algebra with a comultiplication (a normal unital coassociative *-homomorphism $\Delta: M\to M\bar\otimes M$). In this case, the predual $M_*$ is a Banach algebra with the multiplication $(\mu*\nu)(x) = (\mu\otimes\nu)\big(\Delta(x)\big)$, $\mu,\nu\in M_*$, $x\in M$. The space $WAP(M)$ is thus well defined.
The principal example is the space $WAP(\G)$ on a locally compact {\sl quantum} group $\G$, corresponding to the algebra $A=L^1(\G)$ with the dual $M=A^*=L^\infty(\G)$.
If $G$ is a locally compact group, the space is the same as the usual space $WAP(G)$.

Despite the naturality of this definition, the space $WAP(M)$ appears to be difficult to work with, as shown by V.~Runde \cite{runde-wap,runde-factor}: while being self-adjoint, it might not be an algebra; also, there is no clue as to whether it carries invariant means, even in the case of $M=L^\infty(\G)$ for a locally compact quantum group $\G$. An invariant mean here is a positive linear functional $m$ which is invariant under left or right module action of $M_*$.

This motivated M.~Daws \cite{daws-wap} to introduce a more restricted space $wap(M)$ which is a $C^*$-algebra and is contained in $WAP(M)$ (in fact, its maximal $C^*$-subalgebra). There are however no results on invariant means on this space. The largest space known to admit an invariant mean is the Eberlein algebra \cite{daws-das} which is the closure of the space of coefficients of unitary representations of $M_*$. In the commutative case, this algebra is the uniform closure of the Fourier--Stieltjes algebra $B(G)$.

The present work addresses this question by proposing a pair of other, smaller spaces contained in $WAP(M)$ which do carry left and right invariant means respectively. The definition is based on yet another description of the space $WAP(S)$ in the case of a semitopological semigroup $S$, the reflexive representation theorem of M.~Megrelishvili \cite[Theorem 4.6]{megrelishvili}: 
a function $f\in C_b(S)$ is weakly almost periodic if and only if it is a coefficient of a contractive representation of $S$ on a reflexive Banach space. (A more precise citation is given in Section 2.)

In the case of a Hopf--von Neumann algebra $M$, a special r\^ole is player by the unit ball $S$ of $M_*$ which is a  semigroup. Every $x\in M$ defines a function $f_x\in C_b(S)$ by the natural identification $f_x(\mu) = \mu(x)$. One can check that $x\in WAP(M)$ if and only if $f_x\in WAP(S)$. The theorem of Megrelishvili gives therefore a description of $WAP(M)$ as the space of coefficients of contractive reflexive representations of $M_*$. 

Inside $WAP(M)$, we distinguish the subspace of coefficients of so called $\ell_2$-bounded representations (Definition \ref{def-l2-bd}). This is a non-Hilbert analogue of complete boundedness. Further, it contains the space of coefficients of representations having an isometric generator (Definitions \ref{def-generator} and \ref{def-wap-cb-iso}), and it turns out that one has to consider left and right versions of this space, that is, a pair of spaces $WAP_{iso,l}(M)$ and $WAP_{iso,r}(M)$. 

The main result is the existence of invariant means on these spaces. In Section 3, we show that the space $WAP_{iso,l}(M)$ carries a right invariant mean, and on the space $WAP_{iso,r}(M)$, there is a left invariant mean. In general, one cannot say more; but if the comultiplication of $M$ is nontrivial (Definition \ref{def-nontrivial}), each of these means is both left and right invariant.
This is the case in particular if $M=L^\infty(\G)$ where $\G$ is a locally compact quantum group. The spaces in question are so far the largest known to admit invariant means in the quantum context. 


In the case of a classical locally compact group $G$ and $M=L^\infty(G)$, we have $WAP(G) = WAP_{iso}(M)$ (left and right spaces are the same). This is shown in Section 4.


\smallskip

{\bf Acknowledgements.}
This paper would not appear without discussions with Biswarup Das on the importance of invariant means. Our joint work on a close topic will appear elsewhere.

The author is supported by the ANR-19-CE40-0002 grant of the French National Research Agency (ANR). The Laboratory of Mathematics of Besan\c con receives support from the EIPHI Graduate School (contract ANR-17-EURE-0002).

\section{Completely bounded representations and the space $WAP_{iso}(M)$}

From now on, we fix a Hopf--von Neumann algebra $M$, with its comultiplication $\Delta$. In $M_*$, the unit ball $S=\{\mu\in M_*: \|\mu\|\le1\}$ and the set of normal states $P=\{\mu\in S: \mu(1)=1\}$ are semigroups. We consider them with the weak topology, in which they are botn semitopological semigroups.

For any subspace $N\subset M$, $N^{sa}$ is the subset of self-adjoint elements in it. $M$ is called {\it coamenable} if $M_*$ has a bounded approximate identity.

If $V$ is a Banach space, then $V_1$ denotes its unit ball and $B(V)$ the space of bounded linear operators on $V$. For $\eta\in V^*$ and $v\in V$, we write the value of $\eta$ on $v$ as $\langle v,\eta\rangle$.

We also use notations introduced in the Introduction.

\subsection{Correspondence between $WAP(M)$ and $WAP(S)$}

Every $x\in M$ defines a continuous function $f_x$ on~$S$, $f_x(\mu)=\mu(x)$. The set $F=\{f_x: x\in M\}$ is exactly the space of affine functions on $S$ vanishing at 0, and the map $J: x\mapsto f_x$ is isometric and a homeomorphism with respect to the weak topology on $C_b(S)$ and the weak topology in $M$.

The translates of $f_x$ are related to the module action of $M_*$ as $L_\mu(f_x) = f_{x.\mu}$, $R_\mu(f_x) = f_{\mu.x}$, $\mu\in M_*$, so that the orbits $S.x=\{ \mu.x: \mu\in S\}$ and $J(S.x) = Orb_S(f_x)$ are simultaneously relatively weakly compact or not in $M$ and $C_b(S)$ respectively.
It follows that $x\in WAP(M)$ if and only if $f_x\in WAP(S)$. Of course, we do not obtain in this way all weakly almost periodic functions on $S$, but only affine ones vanishing at 0.

It is known \cite{young, daws-reflexive} that $WAP(M)$ is the space of coefficients of representations of $M_*$ on reflexive Banach spaces. A subtlety here is however that this is valid in the case when $W_*$ is unital (at least with a bounded approximate identity), or otherwise requires adding an external unit. This is not essential for the space $WAP(M)$ by itself, but when speaking of ``isometric'' representations (our space $WAP_{iso}(M)$ defined in Section \ref{sec-wap-iso}), the space can change significantly, as shown in Example \ref{trivial-comult}. For this reason we adhere to another, more general description of the space $WAP(M)$ as coefficients of covariant pairs of a representation and a cocycle, given below.

We need an adaptation to our special case of the reflexive representation theorem of Megrelishvili \cite[Theorem 4.6]{megrelishvili}. Apart from the fact that we do not need its full generality, there are two points to change: our spaces are complex, and we should arrive at linear maps $\pi$ and $\xi$. Since the proof is not long, it is easier to rewrite it in full making necessary modifications on the way, rather than commenting on the original one.
 
\begin{thm}\label{left-cov}
Let $M$ be a Hopf--von Neumann algebra and $S = \{ s\in M_*: \|s\|\le1 \}$. For every $x\in WAP(S)$ there exists a reflexive Banach space $V$, a contractive homomorphism $\pi$ from $M_*$ to the space of bounded operators on $V$, and a contractive linear 
map $\xi$ from $M_*$ to $V$ such that the pair $(\pi,\xi)$ is left covariant: $\xi(st)=\pi(s)\xi(t)$, and $x(t) = \langle \xi(t),\eta\rangle$, $t\in S$, with some $\eta\in V^*$.
\end{thm}
\begin{proof}
Denote $\D=\{\lambda\in \C:|\lambda|\le1\}$. Recall that a subset $A$ of a complex vector space is called circled if $\D a\subset A$ for every $a\in A$.

Fix $x\in M$. The weak closure $\overline{x.S}$ of $x.S$ is weakly compact, convex and circled. Next, the convex hull $C_x$ of $\overline{x.S}\cup \D\{x\}$ is weakly compact again (in a Hausdorff topological vector space, the convex hull of a finite family of compact convex sets is compact \cite[II.10.2]{schaefer}). Set $W=J(C_x)$; it is weakly compact as the continuous image of a weak compact. Let $E$ be the (norm) closed linear span of $W$ in $C_b(S)$. $W$ is also weakly compact in $E$ in its weak topology, which is the same as induced by the weak topology of $C_b(S)$. $W$ is also convex and circled, so the Davies--Figiel--Johnson--Pelczynski \cite{DFJP} factorisation procedure applies to it: for $n\in\N$, set $U_n=2^nW+2^{-n}E_1$ and let $\|\cdot\|_n$ be the Minkowski functional of $U_n$: $\|f\|_n = \inf\{\lambda>0: f\in \lambda U_n\}$. Define next for $f\in E$
$$
N(f) = \Big( \sum_{n=1}^\infty \|f\|_n^2\Big)^{1/2}
$$
and $V = \{ f\in E: N(f)<\infty\}$. By \cite{DFJP}, $V$ is a reflexive Banach space such that $W\subset V_1$, and $V$ is continuously embedded into $E$, so that there exists a constant $C$ such that $\|f\|\le CN(f)$ for $f\in V$. This constant can be estimated more exactly: we have $2W+\frac12 E_1\subset (2\|x\|+\frac12)E_1$ so that $\|f\|\le (2\|x\|+\frac12) \|f\|_1$ and $C \le 2\|x\|+\frac12$.

In the notations used already in the introduction, $S$ acts on $C_b(S)$ by left translations: $L_sf(t)=f(st)$ for $s,t\in S$, $f\in C_b(S)$ (note that $L_{st}=L_tL_s$). 
This action is weakly continuous, and
$L_s J(x.t) = J(x.(ts)) \in J(x.S)$. Also, $L_s(Jx) = J(x.s)$. Thus, $W$ is invariant under every $L_s$. The same is true for $E_1$, as $L_s$ is contractive. Now, $L_sU_n\subset U_n$ for every $n$, so $f\in \lambda U_n$ implies $L_sf\in\lambda U_n$ whence $\|L_sf\|_n\le \|f\|_n$. It follows that $L_s$ is contractive on $V$. We set $\pi(s)=L_s^*$, acting on $V^*$.

The semigroup $S$ is embedded into $V^*$ as $\xi(s): f\mapsto f(s)$, with the bound
$|\xi(s)(f)| = |f(s)| \le \|f\|_\infty = \|f\|_E \le C N(f)$. It follows that $\xi(S)\subset C (V^*)_1$. For $f\in V$ and $s,t\in S$, $\langle f,\xi(st)\rangle = f(st) = L_sf(t) = (\pi(s)^*f)(t) = \langle \pi(s)^*f, \xi(t)\rangle = \langle f, \pi(s)\xi(t)\rangle$, so that $\xi(st) = \pi(s)\xi(t)$.

We have $s(x) = (Jx)(s) = \langle \xi(s),Jx\rangle = \langle \xi(s),\eta\rangle$ for every $s\in S$.
Every function in $J(M)$ is affine on $S$ with $Jf(0)=0$. The same is true by continuity for every $f\in E$. For every $f\in V$, we have thus $\langle \xi(s),f\rangle = f(s)$ is affine in $s$ and $f(\xi(0))=0$. It follows that $\xi$ is affine on $S$ and vanishes at 0.
 It follows that $\xi$ extends to a linear (bounded) map from $M_*$ to $V^*$. Similarly, $\pi$ extends to a linear map on $M_*$. It is clearly still a homomorphism.
By construction, $V\subset E\subset J(M)$; this shows that $\xi:M_*\to V^*$ is weakly continuous.

Finally, set $\eta = J(x)\in W\subset V$.  Note that $\|\eta\|\le1$ as $W\subset V_1$, and $\|\xi(s)\|\le C \|s\|$ where $C\le 2\|x\|+\frac12$. 
The statement of the theorem follows after replacing $(\xi,\eta)$ with $(\xi/C,C\eta)$ and then exchanging $V$ and $V^*$.
\end{proof}

\begin{defn}
A {\it left covariant representation} of $M_*$ on a Banach space $V$ is a pair $(\pi,\xi)$ where $\pi$ is a contractive representation of $M_*$ on $V$, $\xi:M_*\to V$ is a contractive linear map, and $\xi(\mu\nu) = \pi(\mu)\xi(\nu)$ for every $\mu,\nu \in M_*$. A {\it right covariant representation} is defined in the same way except that $\xi(\mu\nu) = \pi(\nu)\xi(\mu)$ for every $\mu,\nu \in M_*$.
A {\it coefficient} of $(\pi,\xi)$ is any functional on $M_*$ of the form $\mu\mapsto \langle \xi(\mu),\eta\rangle$, with some $\eta\in V^*$. We identify it with the corresponding  element of $M$.
\end{defn}

The reasoning above, with its obvious converse, shows that $WAP(M)$ is the space of coefficients of all left covariant representations of $M_*$ on reflexive Banach spaces.

One can prove also the right version of Theorem \ref{left-cov} and show that every $x\in WAP(M)$ is also a coefficient of a right covariant representation of $M_*$ on a reflexive Banach space. For this, one starts with the orbit $S.x$ instead of $x.S$, next considers right translations on $C_b(S)$, $R_sf(t) = f(ts)$, having this time $R_{st}=R_sR_t$ and $R_sJx=J(s.x)$. The construction of $V$ and the inclusion $\xi: S\to V^*$ are the same, while $\pi(s)=R_s^*$. One verifies finally that $\xi(st) = \pi(t)\xi(s)$ for all $s,t\in S$.

If $M$ coamenable then there is no difference between left and right covariant pairs and one can assume $\xi$ to be constant. To see this, suppose that $a_i\in M_*$ is a bounded approximate identity, and let $(\pi,\xi)$ be a left covariant representation. The net $(\xi(a_i))$ is contained in $V_1$ which is weakly compact; let $(\xi(a_{i_j}))$ be a subnet weakly converging to $\xi_0\in V_1$. For every $\mu\in M_*$, we have $\mu a_{i_j}\to \mu$, in norm and weakly. By the weak continuity of $\xi$, then
$$
\xi(\mu) = \lim \xi(\mu a_{i_j}) = \lim \pi(\mu) \xi(a_{i_j}) = \pi(\mu)\xi_0.
$$
The same formula is proved in a similar way for a right covariant representation.

Let us now discuss the constants appearing in the proof. Writing $\|\xi\|=\sup_{s\in S} \|\xi(s)\|$, we always have $\|x\|\le \|\xi\| \|\eta\|$ if $x(\mu) = \langle \xi(s),\eta\rangle$, be the pair $(\xi,\eta)$ realized by the  construction above or not. In the theorem, before renorming we have $\|\eta\|\le1$ and $\|\xi\|\le C\le 2\|x\|+1/2$, so that $\|x\|=1$ implies $\|\xi\| \|\eta\| \le 3\|x\|$. We can obtain moreover the last inequality for any $x\in WAP(M)$ by applying the construction to $x/\|x\|$ (and then multiplying $\eta$ by $\|x\|$). Finally, if we set
$$
\|x\|_{coe\!f} = \inf \{ \|\xi\|\,\|\eta\|: \ x= \langle \xi(s),\eta\rangle \}
$$
where the infimum is taken over all $\xi,\eta$ appearing in left covariant representations of $M_*$ on reflexive Banach spaces, then we have:
$$
\frac13 \|x\|_{coe\!f} \le \|x\| \le \|x\|_{coe\!f}.
$$
This might seem unexpected as compared to the norm of the Fourier or Fourier-Stieltjes algebra of a locally compact group $G$ which is defined in a similar way; however this reflects the fact that the space $WAP(G)$ is always uniformly closed.

This is a norm indeed; the triangular inequality follows from the fact that if $x,y\in WAP(M)$ and $(V_x,\pi_x,\xi_x,\eta_x)$,  $(V_y,\pi_y,\xi_y,\eta_y)$ are associated to them, then to $x+y$ one can associate $(V_x\oplus_2 V_y,\pi_x\oplus\pi_y, \xi_x\oplus\xi_y,\eta_x\oplus\eta_y)$, the $\ell_2$-direct sum $V_x\oplus_2 V_y$ being reflexive. Considering right covariant representations, we arrive at the same equivalence.

\def\d{\mathfrak d}

\subsection{$\ell_2$-bounded representations of $M_*$}

The theory of operator spaces \cite{effros-ruan} has proved itself very meaningful in noncommutative harmonic analysis, starting with the theorem of Ruan on the equivalence of the amenability of a locally compact group $G$ to the operator space amenability of its Fourier algebra $A(G)$ \cite{ruan}. One can turn to the survey of Runde \cite{runde-survey} for more on this topic.

An important point is however that we are working with non-Hilbert Banach spaces. Recall that the theory of operator spaces begins with the notion of matrix norms which exist naturally on the space $B(H)$ of bounded operators on a Hilbert space $H$. If the space is non-Hilbert, there is no natural operator space structure on $B(V)$. (In the prepublication version, we collect in the Appendix some calculations showing why several possible definitions fail.) However, certains elements of the theory of operator spaces do work in this context, and we will use them.

\begin{nots}
Let $V$ be a Banach space and $H$ a Hilbert space of dimension $\mathfrak d$ with a basis $\{e_\a\}_{\a\in\mathfrak d}$. We consider the space $\ell_{2,\d}(V)$ as a tensor product:
$$
\ell_{2,\d}(V)  = \big\{ \sum_{\a\in\mathfrak d} e_\a\otimes v_\a: v_\a\in V, \ \sum_\a \|v_\a\|^2<\infty \big\}
$$
with the norm $\|e_\a\otimes v_\a\| = \big(\sum_\a \|v_\a\|^2)^{1/2}$. We denote it also as $H\otwo V$. The notation $x\otimes v$ has an obvious sense for $x\in H$, $v\in V$. If $V$ is a Hilbert space, we get the usual Hilbert space tensor product.

Denote by $\bar H$ the antilinear copy of $H$, and let $h\mapsto \bar h$ be the canonical map from $H$ to $\bar H$. We allow ourselves to denote in the same way the inverse map from $\bar H$ to $H$. If $V$ is reflexive, $H\otwo V$ is also reflexive with the dual $\bar H\otwo V^*$, and note that the duality is linear in both arguments. We will use sometimes the original scalar product on $H$, and if we do, we denote it as $\langle h,k\rangle_H$. For $h,k\in H$, $\xi\in V$, $\eta\in V^*$ we have thus $\langle h\otimes \xi, k\otimes \eta\rangle = \langle h,\bar k\rangle_H \langle \xi,\eta\rangle$. In general, it will be natural to denote vectors in $\bar H\otwo V^*$ as $\bar h\otimes \eta$, $h\in H$, $\eta\in V^*$.

From now on, we suppose that $M$ is a Hopf--von Neumann algebra acting on a Hilbert space $H$ with a basis $\{e_\a\}_{\a\in\mathfrak A}$. For $h,k\in H$, we denote by $\mu_{kh}\in M_*$ the vector functional $\mu_{kh}(x)=\langle xh,k\rangle_H$, $x\in M$.
\end{nots}

It is known that completely bounded represenations of $M_*$ on Hilbert spaces are characterized by their {\it generators} (the definition is identical to the following one), and in the case of quantum groups, often generators themselves are called (co-)representations:

\begin{defn}\label{def-generator}
Let $\pi:M_*\to B(V)$ be a linear map. We say that $\pi$ has a {\it generator} $U\in B(H\otwo V)$ if for every $h,k\in H$, $v\in V$, $\eta\in V^*$
\begin{equation}\label{U-pi}
\langle U(h\otimes v), \bar k\otimes \eta\rangle = \langle \pi(\mu_{k h}) v,\eta\rangle.
\end{equation}
\end{defn}
One can choose different realizations of $M$ as an algebra of operators on a Hilbert space, and this would lead to different generators. It is known that in the case then $V$ is a Hilbert space, $\pi$ is completely bounded if and only if it has a generator. (Note that we assume nothing more than boundedness of $U$.)

In the non-Hilbert case, this notion is also very useful, and our next step it to establish a sufficient condition for the existence of a generator.
 
Let $n$ be a natural number, and let $\ell_{2,n}$ denote the space $\C^n$ with the euclidean norm. The space of matrices $M_n(B(V))$ is linearly isomorphic to the space $B(\ell_{2,n}\otwo V)$, with the isomorphism $\phi_n: M_n(B(V))\to B(\ell_{2,n}\otwo V)$ given by
\begin{equation} \label{phin}
\phi_n(u)(\xi) = \sum_{k,j=1}^n e_j\otimes u_{jk}\xi_k
\end{equation}
for $u=(u_{jk})\in M_n(B(V))$ and $\xi=\sum\limits_{k=1}^n e_k\otimes \xi_k\in \ell_{2,n}\otwo V$. This defines a family of matrix norms on $B(V)$ (which satisfy the first but in general not the second axiom of an operator space). We get the usual operator space structure if $V$ is a Hilbert space.

\begin{defn}\label{def-l2-bd}
Let $E$ be an operator space, $V$ a Banach space, and let $\pi:E\to B(V)$ be a bounded linear map. We say that $\pi$ is {\it $\ell_2$-bounded} if there exists a constant $C\ge0$ such that $\|\big( \pi(u_{ij}) \big)\| \le C\|u\|$ for every $n\in\N$ and $u=(u_{ij})\in M_n(E)$, with the matrix norms on $B(V)$ defined above. If this holds with $C=1$, then $\pi$ is said to be {\it $\ell_2$-contractive}.
\end{defn}

Of course this is the usual definition of a completely bounded map if $V$ is a Hilbert space.

We always consider $M_*$ with the structure dual to that of $M$.

\begin{lm}\label{norm-muab}
For every finite set $A\subset\mathfrak A$ the norm of the matrix $\big( \mu_{e_\a e_\b}\big)_{\a,\b\in A}$ in $M_{n}(M_*)$, $n=|A|$, is at most 1.
\end{lm}
\begin{proof}
This norm is defined by duality \cite[(3.2.3)]{effros-ruan} as
$$
\sup \{ \| (\mu_{e_\a e_\b}(x_{kl})) \|_{M_{n^2}(\C)} : x\in M_n(M), \|x\|\le1\}.
$$
To estimate this norm, fix $x=(x_{kl})\in M_n(M)$ and $\xi\in \C^{n^2}$. We write $\xi=(\xi_{\b l})_{1\le l\le n \atop \b\in A}$ and calculate
\begin{align*}
\| (\mu_{e_\a e_\b}(x_{kl})) (\xi)\|^2_{\ell_{2,n^2}} &= \| (\sum_{\b,l} \mu_{e_\a e_\b}(x_{kl})\xi_{\b l})_{\a \in A\atop 1\le k\le n}\|^2_{\ell_{2,n^2}}
= \| (\sum_{\b,l} \langle x_{kl} e_\b, e_\a\rangle \xi_{\b l})_{\a,k}\|^2_{\ell_{2,n^2}}
\\& = \sum_{\a, k} | \langle \sum_{\b,l} x_{kl} \xi_{\b l} e_\b, e_\a\rangle |^2
 \le \sum_k \| \sum_{\b,l} x_{kl} \xi_{\b l} e_\b \|^2_{H}
\\& = \sum_k \| \sum_{l} x_{kl} (\sum_\b \xi_{\b l} e_\b )\|^2_H
 = \|\big( \sum_l x_{kl} (\sum_\b \xi_{\b l} e_\b) \big)_{1\le k\le n} \|^2_{H^n}
\\& \le \|x\|^2_{M_n(M)} \|\big( \sum_\b \xi_{\b l} e_\b \big)_{1\le l\le n} \|^2_{H^n}
\\& = \|x\|^2_{M_n(M)} \sum_{l=1}^n \| \sum_\b \xi_{\b l} e_\b \|^2_H
\\& = \|x\|^2  \sum_{l=1}^n  \sum_{\b\in A} |\xi_{\b l} |^2
= \|x\|^2 \|\xi\|^2.
\end{align*}
It follows that $\|\big( \mu_{e_\a e_\b}\big)_{\a,\b\in A}\|_{M_{n}(M_*)}\le 1$.
\end{proof}

Suppose now that $U$ is a generator of $\pi:M_*\to B(V)$. 
For any $h\in H, v\in V$ we can decompose $U(h\otimes v) = \sum_\a e_\a\otimes u_\a$. This decomposition can be expressed via $\pi$: for any $\a$ and every $\eta\in V^*$,
$$
\langle U(h\otimes v), \bar e_\a\otimes \eta\rangle = \langle u_\a,\eta\rangle = \langle \pi(\mu_{e_\a h}) v,\eta\rangle,
$$
so that $u_\a=\pi(\mu_{e_\a h}) v$ and
\begin{equation}\label{U-basis}
U(h\otimes v) = \sum_\a e_\a\otimes \pi(\mu_{e_\a h}) v.
\end{equation}

Similarly, if $U^*(\bar k\otimes u) = \sum_\a \bar e_\a \otimes u'_\a$, then
\begin{align*}
\langle v,u'_\a\rangle =  \langle e_\a\otimes v, U^*(\bar k\otimes u) \rangle &= \langle U(e_\a\otimes v), \bar k\otimes u \rangle
 = \langle \pi(\mu_{k e_\a}) v,u\rangle
\\&  = \langle v, \pi(\mu_{k e_\a})^* u\rangle,
\end{align*}
so that $u'_\a = \pi(\mu_{k e_\a})^* u$ and
\begin{equation}\label{U*-basis}
U^*(\bar k\otimes u) = \sum_\a \bar e_\a\otimes \pi(\mu_{k e_\a})^* u.
\end{equation}

\begin{prop}
If a linear map $\pi:M_*\to B(V)$ is $\ell_2$-bounded, then it has a generator $U\in B(H\otwo V)$ of norm $\|U\|\le \|\pi\|_{cb}$.
\end{prop}
\begin{proof}
Suppose that $\pi$ is $\ell_2$-bounded.
Pick $v=\sum_{\a\in A} e_\a\otimes v_\a\in H\otwo V$, where $A$ is finite; we set also $v_\a=0$ for $\a\notin A$. According to \eqref{U-basis}, we should set
\begin{equation} \label{Uv}
Uv = \sum_{\a,\b} e_\b\otimes\pi(\mu_{e_\b e_\a})v_\a.
\end{equation}
For every finite $B\supset A$ of cardinality $n=|B|$ we have, by Lemma \ref{norm-muab}:
\begin{align*}
\|\sum_{\a,\b\in B} e_\b\otimes\pi(\mu_{e_\b e_\a})v_\a\| &
\le \|\big( \pi(\mu_{e_\b e_\a})\big)_{\a,\b\in B}\|_{M_{n}(V)} \|v\|
\\& \le \|\pi\|_{cb} \|\big( \mu_{e_\b e_\a}\big)_{\a,\b\in B}\|_{M_{n}(M_*)} \|v\| \le \|\pi\|_{cb}\|v\|.
\end{align*}
This implies the convergence of the series \eqref{Uv}, with the estimate $\|Uv\| 
\le \|\pi\|_{cb} \|v\|$. Next, $U$ can be extended by continuity to a bounded operator on $H\otwo V$, which satisfies \eqref{U-basis} by construction.
\end{proof}

\subsection{Definitions of $WAP_{cb}(M)$ and $WAP_{iso}(M)$}\label{sec-wap-iso}

\begin{defn}\label{def-wap-cb-iso}
We define the space $WAP_{cb,l}(M)$ ($WAP_{cb,r}(M)$) as the space of coefficients of all left (right) covariant representations of $M_*$ on reflexive Banach spaces in which $\pi$ is $\ell_2$-contractive. Next, $WAP_{iso,l}(M)$ (respectively $WAP_{iso,r}(M)$) is the subspace of $WAP_{cb,l}(M)$ ($WAP_{cb,r}(M)$) consisting of coefficients of left (respectively right) covariant representations for which the generator is isometric onto. If $M$ is coamenable, the left and right spaces are the same and we denote them $WAP_{cb}(M)$ and $WAP_{iso}(M)$ respectively.
\end{defn}

There is no reason to expect in general that these spaces would be norm closed in $M$. But the invariant means constructed later will of course extend to the closure. 

The next example shows that the left/right versions can be very different:

\begin{ex}\label{trivial-comult}
Consider $M=B(H)$ with the trivial comultiplication $\Delta(x) = 1\otimes x$. This induces on $M_*=N(H)$ (the space of trace class operators) the multiplication $\mu\nu=\mu(1)\nu$. One has then $\mu.x=\mu(x)1$ and $x.\mu=\mu(1)x$ for any $x\in M$, $\mu\in M_*$.

Fix $x\in M$ and define $\xi:M_*\to \C$ by $\xi(\mu)=\mu(x)$. This map is left covariant with respect to the trivial representation $\pi(\mu)=\mu(1)$: $\xi(\mu\nu) = (\mu\nu)(x) = \mu(1)\nu(x) = \pi(\mu)\xi(\nu)$. Next, $\pi$ has an isometric generator: with the identity operator $\Id$ on $H\otwo \C$, the identity \eqref{U-pi} holds:
$$
\langle \Id(h\otimes v), \bar k\otimes \eta\rangle
= \langle h,k\rangle_H \langle v,\eta\rangle = \mu_{kh}(1) v\eta = \langle \pi(\mu_{k h}) v,\eta\rangle.
$$
Thus, $x\in WAP_{iso,l}(M)$, so that $WAP_{iso,l}(M)=M$ (and also equal to $WAP(M)$).

 Any functional $\phi$ on $M$ is clearly right invariant, so any positive functional will serve as a right invariant mean. But (if only $M\ne\{\C1\}$) there is no left invariant mean on $M$, as $\phi(\mu.x)=\mu(x)\phi(1)$ is not equal to $\phi(x)\mu(1)$ for $x,\mu$ such that $\phi(x)\ne0$ but $\mu(x)=0\ne\mu(1)$.

Let us suppose now that $x\in M$ is such that $\mu(x) = \langle \pi(\mu)\xi,\eta\rangle$ for every $\mu\in M_*$, where $\pi$ is a representation of $M_*$ on a reflexive Banach space $V$ and $\xi\in V$, $\eta\in V^*$. Let $V_0$ denote the closure of $\pi(M_*)V$. For every $\mu,\nu\in M_*$, $v\in V$ we have $\pi(\mu\nu)v = \pi(\mu)\pi(\nu)v = \mu(1)\pi(\nu)v$, whence $\pi(\mu)=\mu(1)\Id$ on $V_0$. If $x\notin\C1$, we cannot have $\xi\in V_0$, thus $V_0\ne V$.

Suppose now that $\pi$ has an isometric generator $U$. For $v\in V_0$, we have
$$
U(e_\a\otimes v) = \sum_\b e_\b\otimes \pi(\mu_{e_\b e_\a})v = \sum_\b e_\b\otimes \mu_{e_\b e_\a}(1) v = e_\a\otimes v,
$$
so that $U$ acts as the identity on $H\otimes V_0$. On the other hand, $e_\a\otimes\xi \notin H\otimes V_0$, but we have clearly $U(e_\a\otimes\xi)\in H\otimes V_0$. This implies that $U$ cannot be isometric.

This shows that if we defined the space $WAP_{iso}(M)$ without considering covariant pairs, in this case it would equal $\C1$.
\end{ex}

\begin{rem}
It follows from the proof of \cite{daws-das} that the Eberlein algebra $E(\G)$ of a locally compact quantum group $\G$ is the closure in $L^\infty(\G)$ of coefficients of unitary representations of $L^1(\G)$. As a consequence, $E(\G)\subset WAP_{iso}(L^\infty(\G))$.
\end{rem}

\begin{nots}
By definition, for every $x\in WAP_{cb,l}(M)$ (respectively $WAP_{cb,r}(M)$) there exists a reflexive Banach space $V$, an $\ell_2$-bounded left (right) covariant representation $(\pi,\xi)$ of $M_*$ on $V$ and a vector $\eta\in V^*$ such that $\mu(x) = \langle \xi(\mu),\eta\rangle$ for every $\mu\in M_*$. We will say that the quadruple $(V,\pi,\xi,\eta)$ is associated to $x$.

The subspace $V_\xi=\overline{\xi(M_*)}$ is $\pi$-invariant and reflexive, thus we can consider the restriction of $\pi$ onto $V_\xi$ instead. In the sequel, we will always suppose that $V=V_\xi$.
\end{nots}

All these spaces are invariant under left and right actions of $M_*$. Indeed, let for example $(V,\pi,\xi,\eta)$ be associated to $x\in WAP_{cb,l}(M)$. Then $\nu(\mu.x) = \langle \xi(\nu\mu),\eta\rangle = \langle \pi(\nu)\xi(\mu),\eta\rangle$ for $\mu,\nu\in M_*$, so that setting $\xi_\mu(\nu)=\pi(\nu)\xi(\mu)$ we get $\mu.x\in WAP_{cb,l}(M)$ with the quadruple $(V,\pi,\xi_\mu,\eta)$ is associated to it. At the same time, $\nu(x.\mu) = \langle \xi(\mu\nu),\eta\rangle = \langle \xi(\nu), \pi(\mu)^*\eta\rangle$, so that $x.\mu\in  WAP_{cb,l}(M)$ with the quadruple $(V,\pi,\xi,\pi(\mu)^*\eta)$. The representation $\pi$ remains thus the same, and it is $\ell_2$-bounded or with isometric generator respectively.

\begin{nots}\label{map-R}
Fix $x\in WAP(M)$, and let $(V,\pi,\xi,\eta)$ be associated with it, with $(\pi,\xi)$ being left covariant. As in the classical case \cite{godement}, we can define a continuous map $L:V\to M$ as follows. Set first $L\big(\xi(\mu)\big) = \mu.x$ for $\mu\in M_*$. If $\xi(\mu)=\xi(\nu)$, then for every $\lambda\in M_*$ we have $\lambda(\mu.x) = (\lambda\mu)(x) = \langle \xi(\lambda\mu),\eta\rangle = \langle \pi(\lambda) \xi(\mu),\eta\rangle = \langle \pi(\lambda) \xi(\nu),\eta\rangle = \lambda(\nu.x)$, so that $\mu.x=\nu.x$ and $L$ is well defined. Moreover, by the same calculations
$$
\| L\big(\xi(\mu)\big) \| = \sup_{\lambda\in S} |\lambda(\mu.x)| = \sup_{\lambda\in S} |\langle \pi(\lambda)\xi(\mu), \eta\rangle| \le \|\eta\|\,\|\xi(\mu)\|.
$$
It follows that $L$ extends to $V$ by continuity (recall that we suppose that $V$ is the closure of $\xi(M_*)$).
In a similar way, we define a map $R':V^*\to M$ by $R'(\pi(\mu)^*\eta)=x.\mu$ and check that it is well defined and extends by continuity to the closure of $\pi(M_*)^*V^*$ (we need nothing more from it).

If the pair $(\pi,\xi)$ is right covariant, two other maps are defined: $R:V\to M$ by $R\big(\xi(\mu)\big) = x.\mu$ and $L':\pi(M_*)^*V^*\to M$ by $L'(\pi(\mu)^*\eta)=\mu.x$, $\mu\in M_*$.
\end{nots}

It is unclear whether $WAP_{cb}(M)$ is closed under multiplication. It is known that $WAP(M)$ is not, in general; the article \cite{runde-factor} gives some partial results and can show the level of complexity of the question.


\section{Invariant means on $WAP_{iso}(M)$}

\subsection{Existence of constants in $P$-orbits}

Recall that we denote by $P$ the set of normal states of $M$, that is, $P=\{\mu\in S: \mu(1)=1\}$. This is a subsemigroup in $S$.

Note that $P.x$ is convex for every $x\in M$, so that its norm and weak closures are the same. We denote this closure by $\overline{P.x}$. The same applies to $\overline{x.P}$.

\begin{prop}\label{cxP}
For every $x\in WAP_{iso,l}(M)$ there is a constant $c\in\C$ such that $c1$ is in $\overline{P.x}$.
\end{prop}
\begin{proof}
Let $K_0$ be the norm closure of $\xi(P)$, and let $K$ be the subset of vectors of minimal norm in $K_0$. It is nonempty since $V$ is reflexive and $K_0$ is convex and bounded. By definition, $K_0$ is $\pi(P)$-invariant. Since $\pi(\mu)$ is contractive for every $\mu\in P$, it sends $K$ to itself and preserves the norm of every $v\in K$. $K$ is convex, bounded and norm closed, so it is weakly compact.

If $v\in K$, then for every $\a$ we have by \eqref{U-basis} $U(e_\a\otimes v) = \sum_\beta e_\beta\otimes \pi(\mu_{e_\beta e_\a}) v$, with the norm
$$
\|U(e_\a\otimes v)\| = \Big(\sum_\beta \|\pi(\mu_{e_\beta e_\a}) v\|^2 \Big)^{1/2}.
$$
Since $U$ is isometric, this is equal to $\|e_\a\otimes v\| = \|v\|$, and as $\mu_{e_\a e_\a}\in P$, this equals also to $\|\pi(\mu_{e_\a e_\a}) v\|$. It follows that 
 $\pi(\mu_{e_\beta e_\a}) v=0$ if $\beta\ne\a$.

The maps $U_\a=\pi(\mu_{e_\a e_\a})$ are thus isometric on $K$: if $u,v\in K$ then
$$
\|U_\a(u-v)\| = \|U(e_\a\otimes (u-v))\| = \|e_\a\otimes (u-v)\| = \|u-v\|.
$$
We are therefore in the assumptions of the Ryll-Nardzewsky theorem \cite[Theorem A.24]{burckel} with the semigroup generated by $(U_\a)$, and it follows that $(U_\a)$ have a common fixed point $\tilde\xi \in K$.

For every $\a$ we have thus $U(e_\a\otimes\tilde\xi) = e_\a\otimes\tilde\xi$, and therefore $U(u\otimes\tilde\xi)=u\otimes\tilde\xi$ for every $u\in H$. This implies that for every $v\in H$, $\eta\in V^*$
$$
\langle \pi(\mu_{vu})\tilde\xi,\eta\rangle = \langle U(u\otimes\tilde\xi),\bar v\otimes\eta\rangle = \langle u\otimes\tilde\xi,\bar v\otimes\eta\rangle
= \langle u, \bar v\rangle \langle \tilde\xi,\eta\rangle = \mu_{vu}(1) \langle \tilde\xi,\eta\rangle,
$$
so that $\pi(\mu_{vu})\tilde\xi = \mu_{vu}(1) \tilde\xi$, which implies that $\pi(\mu)\tilde\xi = \mu(1) \tilde\xi$ for general $\mu\in M_*$. In particular, $\tilde\xi$ is a fixed vector of $\pi(P)$.

Let $L$ be the map defined in Notations \ref{map-R}. We show next that $L(\tilde\xi)\in \C1$. Pick a sequence $(\mu_n)\subset M_*$ such that $\tilde\xi = \lim \xi(\mu_n)$. For every $\mu\in M_*$ we have
\begin{align*}
\mu(L\tilde\xi) &= \lim \mu\big( L\xi(\mu_n)\big) = \lim \mu(\mu_n.x) =  \lim \langle \xi(\mu\mu_n), \eta\rangle = \lim \langle \pi(\mu)\xi(\mu_n), \eta\rangle 
\\&= \langle \pi(\mu)\tilde\xi, \eta\rangle = \mu(1) \langle \tilde\xi, \eta\rangle,
\end{align*}
so that $L\tilde\xi = \langle \tilde\xi, \eta\rangle 1$.

By construction, $\tilde\xi\in \overline{\xi(P)}$, so that $L\tilde\xi\in \overline{P.x}$.
In particular, $\|L\tilde\xi\|\le \|x\|$.
\end{proof}

It is clear that if $x\in WAP_{iso,r}(M)$ then in the same way one shows that there is a constant $c\in\C$ such that $c1\in \overline{x.P}$.

\begin{defn}\label{def-nontrivial}
Let $M$ be a Hopf--von Neumann algebra. Say that its comultiplication $\Delta$ is {\it left (right) nontrivial } if for $x\in M$ the identity $\Delta(x)=x\otimes1$ (respectively $\Delta(x)=1\otimes x$) implies $x\in\C1$.
As in the case of a locally compact quantum group, we say that a Hopf--von Neumann algebra $M$ is {\it coamenable} if $M_*$ has a bounded approximate identity. This condition is always satisfied if $M=L^\infty(G)$ with a locally compact group $G$; for the group von Neumann algebra $M=VN(G)$ it is equivalent to $G$ being amenable.
\end{defn}

\begin{prop}
If $M$ is coamenable, then its comultiplication is left and right nontrivial.
\end{prop}
\begin{proof}
Let $(u_i)$ be a bounded approximate unit in $M_*$. Suppose that $x\in M$ and $\Delta(x) = x\otimes1$. For every $\mu\in M_*$, we have $\mu(x) = \lim_i (u_i\mu)(x) = \lim_i (u_i\otimes \mu)(\Delta(x)) = \lim_i u_i(x) \mu(1)$. In particular, the limit $c=\lim_i u_i(x)$ does exist, and one obtains $x=c1$. The right nontriviality is proved similarly.
\end{proof}

\begin{prop}\label{cPx}
Suppose that the comultiplication of $M$ is right nontrivial. Then for every $x\in WAP_{iso,l}(M)$ there exists a constant $c\in\C$ such that $c1$ is in $\overline{x.P}$.
\end{prop}
\begin{proof}
We can start similarly: Let $K_0$ be the norm closure of $\pi(P)^*\eta$ in $V^*$, and let $K$ be the subset of minimal norm in $K_0$. For every $\mu\in P$, the map $\pi(\mu)^*$ sends $K$ to itself and preserves the norm on it.

Using the decomposition $U^*(\bar e_\a\otimes u) = \sum_\beta \bar e_\beta\otimes \pi(\mu_{e_\a e_\beta})^* u$ dual to \eqref{U-basis} and the fact that $U^*$ is isometric, we conclude in the same way that the maps $\bar U_\a = \pi(\mu_{e_\a e_\a})^*$ are isometric on $K$. Applying again the Ryll-Nardzewski theorem to the semigroup generated by $\bar U_\a$, we obtain a common fixed point $\tilde\eta\in K$, for which $U^*(\bar u\otimes \tilde\eta) = \bar u\otimes \tilde\eta$ for any $u\in H$.

It follows that for every $u,v\in H$ and $\nu\in M_*$
$$
\langle \xi(\nu),\pi(\mu_{vu})^*\tilde\eta\rangle
=  \langle U(u\otimes\xi(\nu)),\bar v\otimes\tilde\eta\rangle
=  \langle u\otimes\xi(\nu),U^*(\bar v\otimes\tilde\eta)\rangle
= \langle u,v\rangle_H \langle \xi(\nu),\tilde\eta\rangle = \mu_{vu}(1) \langle \xi(\nu),\tilde\eta\rangle.
$$
By density of $\xi(M_*)$ in $V$ and by density of the linear span of $(\mu_{uv})$ in $M_*$ it follows that $\pi(\mu)^*\tilde\eta = \mu(1) \tilde\eta$ for any $\mu\in M_*$.

Proving that $R'\tilde\eta$ is a constant (with the map $R'$ defined in Notations \eqref{map-R}) is more delicate than for the left orbit. Let $(\mu_n)\subset P$ be a sequence such that $\pi(\mu_n)^*\eta \to \tilde\eta$; it exists because $\tilde\eta$ is in the closure of $\pi(P)^*\eta$. For every $\mu\in M_*$ we have
$$
\mu(R'\tilde\eta) = \lim \mu(R\big( \pi(\mu_n)^*\eta\big) = \lim \mu(x.\mu_n) = \lim (\mu_n\mu)(x) 
 = \lim \langle \pi(\mu_n)\xi(\mu),\eta\rangle = \langle \xi(\mu),\tilde\eta\rangle.
$$
For every $\mu,\nu\in M_*$ then
$$
\mu\big( (R'\tilde\eta).\nu \big) = (\nu\mu)(R'\tilde\eta) =  \langle \xi(\nu\mu),\tilde\eta\rangle = \langle \xi(\mu),\pi(\nu)^*\tilde\eta\rangle =  \langle \xi(\mu),\nu(1)\tilde\eta\rangle
 = \mu\big( \nu(1) R'\tilde\eta\big),
$$
which implies $(R'\tilde\eta).\nu = \nu(1) R'\tilde\eta$ and $\Delta(R'\tilde\eta) = 1\otimes R'\tilde\eta$.
By assumption, there exists then $c\in\C$ such that $R'\tilde\eta=c1$. 
\end{proof}

Similarly, one proves that if the comultiplication of $M$ is left nontrivial then for every $x\in WAP_{iso,r}(M)$ there exists a constant $c\in\C$ such that $c1$ is in $\overline{P.x}$.

\subsection{Invariant means}

We prove below the existence of a right invariant mean on $WAP_{iso,l}(M)$, and under the condition of right nontriviality of the comultiplication, its left invariance as well. This asymmetry can appear indeed, as shown in Example \ref{trivial-comult}.

The following proof is close to the classical one \cite[Theorem 1.25]{burckel}. 

\begin{lm}\label{Px-properties}
For every $x,y\in M$ the following inclusions hold:
\begin{enumerate}
\item If $x\ge0$ then every constant in $\overline{P.x}$ is nonnegative;
\item If $x=x^*$ then every constant in $\overline{P.x}$ is real;
\item If $x,y\in WAP(M)$, then $\overline{P.(x+y)}\subset \overline{P.x}+\overline{P.y}$.
\item If $x\in WAP(M)$ and $\mu\in P$, then the only possible constant in $\overline{P.(x.\mu-x)}$ is zero.
\end{enumerate}
\end{lm}
\begin{proof}
\begin{enumerate}
\item If $x\ge0$ then $\nu(\mu.x) = (\nu\mu)(x)\ge0$ for every $\mu,\nu\in P$, so that $\mu.x\ge0$. It follows that $\overline{P.x}\subset M_+$.
\item Let $x_+, x_-$ be the positive and negative part of $x$ respectively. As in (1), we notice that $\mu.x_\pm$ is positive on $P$ for every $\mu$, so that $\mu.x$ is real-valued. If $c1\in \overline{P.x}$, then the constant function $c$ on $P$ is the pointwise limit of a net of functions $\mu_\a.x$ with $\mu_\a\in P$, and it is real-valued as well.
\item If $z\in \overline{P.(x+y)}$ then there exists a net $(\mu_\a)\subset P$ such that $\mu_\a.(x+y)\to z$. By weak compactness of $\overline{P.x}$, passing to a subnet if necessary, we can assume that $\mu_\a.x$ converges to $u\in \overline{P.x}$. Now $\mu_\a.y = \mu_\a.(x+y)-\mu_\a.x \to z-u\in \overline{P.y}$.
\item Suppose that $(\mu_\a)$ is a net in $P$ such that $\mu_\a.(x.\mu-x)\to c1$, $c\in\C$. Passing to a subnet if necessary, we can assume that $\mu_\a.x\to y\in \overline{P.x}$. Next, $\mu_\a.(x.\mu) = (\mu_\a.x).\mu\to y.\mu$, which implies $y.\mu-y=c1$. Moreover,
$$
y.\mu^k-y = \sum_{j=1}^k (y.\mu-y).\mu^{j-1} = kc1,
$$
while the norm of this expression is bounded by $\|\mu^k.y\|+\|y\| \le 2\|x\|$. We conclude that $c=0$.
\end{enumerate}
\end{proof}

Denote $WAP_{iso,l}^{sa}(M)=\{ x\in WAP_{iso,l}(M): x=x^*\}$.
Now we can set, for $x\in WAP_{iso,l}^{sa}(M)$,
$$
p(x) = \sup\{ c\in\R : c1\in \overline{P.x}\},
$$
knowing by Proposition \ref{cxP} that this set of constants is nonempty.

We have the following properties:
\begin{lm}\label{p-properties}
For every $x, y\in WAP_{iso,l}^{sa}(M)$, 
\begin{enumerate}
\item $p(cx)=cp(x)$, $c\ge0$;
\item $|p(x)|\le \|x\|$;
\item $p(x+y)\le p(x)+p(y)$;
\item $p(x)\ge0$ if $x\ge0$.
\end{enumerate}
\end{lm}
\begin{proof}
\begin{enumerate}
\item Obvious.
\item For every $\mu\in P$, we have $\|\mu.x\|\le\|\mu\|\,\|x\|=\|x\|$.
\item If $c1\in \overline{P.(x+y)}$, $c\in\R$, then by Lemma \ref{Px-properties} $c1=u+v$ with $u\in \overline{P.x}$, $v\in \overline{P.y}$. Let $(\mu_a)\subset P$ be a net such that $\mu_a.u\to d1\in \R1$. Then $\mu_\a.v \to (c-d)1\in \overline{P.y}$ so that $c1 = (c-d)1+d1\in \overline{P.x} + \overline{P.y}$.
\item Follows from Lemma \ref{Px-properties}(1).
\end{enumerate}
\end{proof}

\begin{thm}
On $WAP_{iso,l}(M)$ there exists a right invariant mean. If the comultiplication of $M$ is right nontrivial, then this mean is also left invariant.
\end{thm}
\begin{proof}
Exactly as in the classical case \cite[Theorem 1.25]{burckel}, the properties in Lemma \ref{p-properties} allow us to apply a version of the Hahn-Banach theorem to obtain a real-valued linear functional $m$ on $WAP_{iso,l}^{sa}(M)$ such that $m(c1)=c$, $c\in\R$, and $m(x)\le p(x)$ for any $x\in WAP_{iso,l}^{sa}(M)$. We extend then $m$ to $WAP_{iso,l}(M)$ by linearity.

By Lemma \ref{Px-properties}(1), if $x\ge0$ then $p(-x)\le0$ so that $m(-x)\le0$, which implies $m(x)\ge0$. Thus, $m$ is positive.

By Lemma \ref{Px-properties}(4), for every $x\in WAP_{iso,l}^{sa}(M)$ and $\mu\in P$ we have $p(x.\mu-x)=0$. This implies that $m(x.\mu)\le m(x)$ and similarly $m((-x).\mu)\le m(-x)$ whence $m(x.\mu)\ge m(x)$. Together, these imply the right invariance of $m$ on $WAP_{iso,l}^{sa}(M)$, and by linearity on $WAP_{iso,l}(M)$.

If now $\Delta$ is right nontrivial, then by Proposition \ref{cPx} for every $x\in WAP_{iso,l}(M)$ there exists a constant $c\in\C$ such that $c1\in\overline {x.P}$. Let $d\in\C$ be such that $d\in \overline{P.x}$. Since $\overline{P.x}$ and $\overline{x.P}$ are norm closures, we can choose sequences $(\mu_n),(\nu_m)$ in $P$ such that $\mu_n.x\to d1$, $x.\nu_m\to c1$ in norm. For $\e>0$, let $n,m$ be such that $\|\mu_n.x-c1\|<\e$ and $\|\nu_m.x-c1\|<\e$. Then $\|\mu_n.x.\nu_m-c\|<\e$ and $\|\mu_n.x.\nu_m-d\|<\e$; $\e$ being arbitrary, it follows that $c=d$. Thus, there is a unique  constant in $\overline{x.P}\cup \overline{P.x}$, which is clearly $m(x)$. Since $\overline{P.(\mu.x)}\subset \overline{P.x}$ for any $\mu\in P$, the unique constant in $\overline{P.(\mu.x)}$ is also $m(x)$, so that $m$ is also left invariant.
\end{proof}

Of course there is a symmetric theorem which states that $WAP_{iso,r}(M)$ admits a left invariant mean, and if the comultiplication of $M$ is left nontrivial, this mean is also right invariant. Moreover, these means extend by continuity to the norm closure of the respective spaces in $M$.

It is well known that in $L^\infty(\G)$, where $\G$ is a locally compact quantum group, the comultiplication is both left and right nontrivial, so that the spaces $WAP_{iso,l}(L^\infty(\G))$ and $WAP_{iso,r}(L^\infty(\G))$ have two-sided invariant means.

\section{The commutative case}

In this section we define another space $WAP_{cb*}(M)$. It has sense only in the presence of an involution on $M_*$, so for example, it is not defined on $C_b(S)^{**}$ for a general semigroup $S$. In return, every locally compact quantum group, and moreover its universal version do fall into this category. Maybe most general class of such algebras is the class of quantum semigroups with involution \cite{qsi}, but we do not go here up to that generality. The space $WAP_{cb*}(M)$ is close to $WAP_{cb}(M)$, which is shown by the fact that they are trivially equal if $M=L^\infty(\G)$ is a Kac algebra.

\subsection{Automatic $\ell_2$-boundedness}

Let $M$ be a commutative Hopf--von Neumann algebra. We consider as always $M_*$ with the dual structure of an operator space, which is the maximal structure \cite{effros-ruan}. It is known that every bounded map from $M_*$ into an operator space is automatically completely bounded. As $B(V)$ is not an operator space in general, we have to reprove the same fact in our context.

\begin{prop}
Let $V$ be a Banach space and $E$ an operator space with the maximal structure. Every bounded map from $E$ to $B(V)$ is $\ell_2$-bounded.
\end{prop}
\begin{proof}
We should check the condition in Definition \ref{def-l2-bd}: there exists a constant $C\ge0$ such that $\|\big( \pi(u_{ij}) \big)\| \le C\|u\|$ for every $n\in\N$ and $u=(u_{ij})\in M_n(E)$. The norm of $\pi(u)$ is the operator norm on $\ell_{2,n}\otwo V$:
\begin{align*}
\|\pi(u)\| &= \sup_{\|\xi\|=1} \| \sum_{k,j=1}^n e_j\otimes \pi(u_{jk})\xi_k \|
 = \sup_{\|\xi\|=1, \|\eta\|=1} \sum_{k,j=1}^n \langle \pi(u_{jk})\xi_k, \eta_j \rangle,
\end{align*}
where we write $\xi=\sum\limits_{k=1}^n e_k\otimes \xi_k\in \ell_{2,n}\otwo V$ and $\eta=\sum\limits_{k=1}^n e_k\otimes \eta_k\in \ell_{2,n}\otwo V^*$.
Fix $\xi,\eta$ and set $\tilde\xi_k=\xi_k/\|\xi_k\|$ (or 0 if $\xi_k=0$), and $\tilde\eta_k=\eta_k/\|\eta_k\|$ (respectively 0 if $\eta_k=0$). 
Define now a map $f: E\to M_n(\C)$ by
$$
f(a)_{kl} = \langle \pi(a) \tilde \xi_k,\tilde\eta_l\rangle.
$$
It is bounded, and for $a\in E$
$$
\|f(a)\| = \sup_{\|x\|=\|y\|=1} \sum_{k,l} \langle \pi(a) \tilde \xi_k,\tilde\eta_l\rangle x_k \bar y_l
\le \|\pi\| \,\|a\| \sup_{\|x\|=\|y\|=1} \|(\tilde \xi_k x_k)\| \,\|(\tilde\eta_l\bar y_l)\|
= \|\pi\| \,\|a\| \,\|\xi\| \,\|\eta\|.
$$
It follows that $f$ is completely bounded. For $u\in M_n(E)$ we get, using the notations of \cite{effros-ruan} where $f_n(u)$ is the matrix with coefficients $(f(u_{ij})_{kl})$, the inequality $\|(f_n(u_{ij})\| \le \|\pi\| \,\|\xi\| \,\|\eta\|\, \|u\| = \|\pi\|\,\|u\|$. Now we can use this to estimate the norm of $\pi(u)$ in $M_{n^2}(\C)$. Define $x,y\in \C^{n^2}$ by  $x_{ik} = \delta_{ik} \|\xi_k\|$ and $y_{jl} = \delta_{jl} \|\eta_l\|$. Then $\|x\|=\|\xi\|=1$ and $\|y\|=\|\eta\|=1$, and
\begin{align*}
\sum_{k,j=1}^n \langle \pi(u_{jk})\xi_k, \eta_j \rangle &= \sum_{k,j=1}^n \langle \pi(u_{jk})\tilde\xi_k, \tilde\eta_j \rangle \|\xi_k\| \,\|\eta_l\| = \sum_{i,j,k,l=1}^n \langle \pi(u_{ij})\tilde\xi_k, \tilde\eta_l \rangle x_{ik} \bar y_{jl}
\\ & = \sum_{i,j,k,l=1}^n f(u_{ij})_{kl} x_{ik} \bar y_{jl}
\le \|f_n(u) \| \|x\| \|y\| \le \|\pi\| \,\|u\|. 
\end{align*}
This proves the proposition.
\end{proof}

\begin{cor}
If $M$ is commutative then $WAP(M)=WAP_{cb}(M)$.
\end{cor}


\subsection{Multiplication in $WAP_{cb}(M)$}

As it was said before, it is unknown whether $WAP_{cb}(M)$ is a subalgebra in $M$.
In the commutative case this is however true, by the equality of this space to $WAP(M)$. In addition to this fact, we will need an explicit description of the arising representations and their generators.

Let $M$ be commutative and coamenable, properties satisfied by $L^\infty(G)$ and $C_0(G)^{**}$ (the dual of the measure algebra) of a locally compact group $G$. Fix $x,y\in WAP_{cb}(M)$, and let $(V_i,\pi_i,\xi_i,\eta_i)$, $i=1,2$, be associated to $x$ and $y$ respectively. By assumption, $\xi_1$ and $\xi_2$ are constant.

Intuitively, the representation corresponding to $xy$ is $\pi=(\pi_1\otimes\pi_2)\Delta_{M_*}$. But the difference of available tensor products does not allow us to give a strict sense to this equality. Instead, we can construct a generator for $\pi$, similarly to \cite{daws-reflexive}, and then deduce $\pi$ from it. All necessary properties are then to be obtained by coordinate calculations.

Let $\otimes_r$ be a tensor product which preserves reflexivity, for example one of the Chevet-Saphar tensor products \cite{ryan}. We set $V=V_1\otimes_{r} V_2$. The maps $\tilde\pi_1=\pi_1\otimes 1$ and $\tilde\pi_2=1\otimes\pi_2$ are bounded homomorphisms from $M_*$ to $B(V)$. They have therefore generators $\tilde U_1,\tilde U_2 \in B(H\otwo V)$. Set $U=\tilde U_1\tilde U_2$, $\xi=\xi_1\otimes\xi_2$, $\eta=\eta_1\otimes\eta_2$. One defines first a map $\pi$ on $B(H)_*$ via $U$ by Definition \ref{def-generator} (extended by linearity and continuity). This gives, for $k,h\in H$, $v_i\in V_i$, $\phi_i\in V_i^*$, $i=1,2$:
\begin{align*}
\langle \pi(\mu_{kh})(v_1\otimes v_2), \phi_1\otimes \phi_2\rangle
&= \langle U(h\otimes v_1\otimes v_2), \bar k\otimes\phi_1\otimes \phi_2\rangle
= \langle \tilde U_2(h\otimes v_1\otimes v_2), \tilde U_1^*(\bar k\otimes\phi_1\otimes \phi_2)\rangle 
\\& = \langle \sum_\a e_\a\otimes \tilde \pi_2(\mu_{e_\a h}) (v_1\otimes v_2), 
 \sum_\b \bar e_\b\otimes \tilde\pi_1(\mu_{k e_\b})^*(\phi_1\otimes \phi_2)\rangle 
\\& = \langle \sum_\a e_\a\otimes v_1\otimes \pi_2(\mu_{e_\a h}) v_2, 
 \sum_\b \bar e_\b\otimes \pi_1(\mu_{k e_\b})^* \phi_1\otimes \phi_2\rangle 
\\& = \sum_\a  \langle \pi_1(\mu_{k e_\a}) v_1, \phi_1\rangle  \langle \pi_2(\mu_{e_\a h}) v_2, \phi_2 \rangle.
\end{align*}
If we denote $x_i(\mu) = \langle \pi_i(\mu) v_i, \phi_i \rangle$, $i=1,2$, then we get $x_{1,2}\in M$ and
$$
\langle \pi(\mu_{kh})(v_1\otimes v_2), \phi_1\otimes \phi_2\rangle =
 \sum_\a \mu_{k e_\a}(x_1) \mu_{e_\a h}(x_2) = \mu_{kh}(x_1 x_2)
$$
(the last equality follows from the scalar product decomposition in $H$), which shows in particular that this expression defines indeed a functional on $M_*$ and not just on $B(H)_*$, so that the map $\pi:M_*\to B(V)$ is well defined. Moreover, the equality 
$$
\langle \pi(\mu)(v_1\otimes v_2), \phi_1\otimes \phi_2\rangle
 = \mu(x_1 x_2)
$$
holds for every $\mu\in M_*$.

Next, fix $a,b,c,d\in H$. Since $v_{12,}$, $\phi_{1,2}$ above are arbitrary, we can decompose below $\pi(\mu_{cd})$ and $\pi(\mu_{ab})^*$ in weakly absolutely converging series and obtain
\begin{align*}
\langle \pi(\mu_{ab}) \pi(\mu_{cd}) & (v_1\otimes v_2), \phi_1\otimes \phi_2\rangle
 = \langle \pi(\mu_{cd})(v_1\otimes v_2), \pi(\mu_{ab})^* (\phi_1\otimes \phi_2)\rangle
\\&= \langle \sum_\a  \big( \pi_1(\mu_{c e_\a}) \otimes \pi_2(\mu_{e_\a d}) \big) (v_1\otimes v_2),
\sum_\b  \big( \pi_1(\mu_{a e_\b}) \otimes \pi_2(\mu_{e_\b b}) \big)^* (\phi_1\otimes \phi_2) \rangle
\\&= \langle \sum_{\a,\b}  \big( \pi_1(\mu_{a e_\b}) \pi_1(\mu_{c e_\a}) \otimes \pi_2(\mu_{e_\b b})\pi_2(\mu_{e_\a d}) \big) (v_1\otimes v_2), \phi_1\otimes \phi_2 \rangle
\\&= \langle \sum_{\a,\b}  \big( \pi_1(\mu_{a e_\b} \mu_{c e_\a}) \otimes \pi_2(\mu_{e_\b b} \mu_{e_\a d}) \big) (v_1\otimes v_2), \phi_1\otimes \phi_2 \rangle
\\& = \sum_{\a,\b} (\mu_{a e_\b} \mu_{c e_\a}) (x_1) (\mu_{e_\b b} \mu_{e_\a d}) (x_2)
\\& = \sum_{\a,\b} (\mu_{a e_\b} \otimes \mu_{c e_\a}) \big( \Delta(x_1) \big)
 (\mu_{e_\b b} \otimes \mu_{e_\a d}) \big( \Delta(x_2) \big)
\\& = \sum_{\a,\b} \langle \Delta(x_1) (e_\b \otimes e_\a),  a\otimes c \rangle
\langle  \Delta(x_2) ( b\otimes d), e_\b \otimes e_\a \rangle
\\&= \langle  \Delta(x_2) ( b\otimes d), \Delta(x_1)^* (a\otimes c) \rangle
= \langle  \Delta(x_1x_2) ( b\otimes d), a\otimes c \rangle
\\& = (\mu_{ab} \otimes \mu_{cd}) \big( \Delta(x_1x_2) \big)
= (\mu_{ab}\mu_{cd}) (x_1x_2) = \langle \pi( \mu_{ab}\mu_{cd} )(v_1\otimes v_2), \phi_1\otimes \phi_2\rangle.
\end{align*}
It follows that $\pi$ is a homomorphism.

\subsection{Equality $WAP(M)=WAP_{iso}(M)$}

\def\G{\mathbb{G}}

Though the final result of this section is proved only for the algebra $L^\infty(G)$ 
 on a locally compact group $G$, certain parts are valid in more generality.

Suppose that $M$ is a Hopf--von Neumann algebra such that an involution is defined on a subalgebra $M_{*\sharp}$ of $M_*$ containing all $\mu_{e_\a e_\b}$. This is the case if $M$ is a quantum semigroup with involution \cite{qsi}, but the main examples are $L^\infty(\G)$ and $C_0(\G)^{**}$ on a locally compact {\sl quantum } group $\G$.

Call an $\ell_2$-contractive representation $\pi$ of $M_*$ {\it admissible} if the map $\bar\pi: \mu\mapsto \pi(\mu^*)^*$ extends from $M_{*\sharp}$ to an $\ell_2$-contractive homomorphism on $M_*$. Let $WAP_{cb*}(M)$ be the space of coefficients of admissible reflexive representations of $M_*$.

This is close to the notion of admissibility of unitary representations of quantum groups as introduced by So\l tan \cite{soltan-bohr}. The question of whether every finite-dimensional unitary representation is admissible is now known as the admissibility conjecture and has not yet a final answer \cite{daws-bohr, das-daws-salmi}.

Suppose that $M$ is coamenable, $\pi$ is admissible, with the generator $U$, and let $\bar U$ be the generator of $\bar\pi$. For every $\xi\in V$, $\eta\in V^*$ we have a decomposition in an absolutely converging series
\begin{align}\label{UUbarveta}
\langle U(e_\a\otimes \xi), \bar U(e_\beta\otimes\eta)\rangle
&= \langle \sum_\gamma e_\gamma\otimes \pi(\mu_{e_\gamma e_\a}) \xi, \sum_\zeta e_\zeta\otimes \bar \pi(\mu_{e_\zeta e_\b}) \eta \rangle
\\& 
 = \sum_\gamma \langle \bar\pi(\mu_{e_\gamma e_\beta})^* \pi(\mu_{e_\gamma e_\a}) \xi,  \eta \rangle
 = \sum_\gamma \langle \pi(\mu_{e_\gamma e_\beta}^*\mu_{e_\gamma e_\a}) \xi,  \eta \rangle,
 \notag
\end{align}
with the estimate
\begin{align*}
\sum_\gamma |\langle \pi(\mu_{e_\gamma e_\beta}^*\mu_{e_\gamma e_\a}) \xi,  \eta \rangle|
&\le
 \|U(e_\a\otimes \xi)\|\, \|\bar U(e_\beta\otimes\eta)\| \le \|\xi\|\,\|\eta\|.
\end{align*}
If $x\in M$ is the corresponding coefficient of $\pi$, $\mu(x) = \langle \pi(\mu)\xi,\eta\rangle$, then
\begin{align*}
\sum_\gamma \langle \pi(\mu_{e_\gamma e_\beta}^*\mu_{e_\gamma e_\a}) \xi,  \eta \rangle
= \sum_\gamma (\mu_{e_\gamma e_\beta}^*\mu_{e_\gamma e_\a})(x).
\end{align*}
In this form, this expression does not depend on the realization of $x$, so we can set, for $x\in WAP_{cb*}(M)$,
\begin{equation}\label{hab}
h_{\a\b}(x) = \sum_\gamma \big( \mu_{e_\gamma e_\beta}^* \mu_{e_\gamma e_\a} \big) (x),
\end{equation}
the series converging absolutely. If we set $\|x\|_{cb} = \inf\|\xi\|\,\|\eta\|$, infimum taken over all $\ell_2$-bounded realizations of $x$ (this is a norm, similarly to $\|\cdot\|_{coe\!f}$), then we have $|h_{\a\b}(x)| \le \|x\|_{cb}$.

Similarly, we can define $h'_{\a\b}$ by
\begin{align*}
\langle \bar U^*(e_\beta\otimes  \xi), U^*(\bar e_\a\otimes \eta)\rangle
&= \langle \sum_\zeta e_\zeta\otimes \bar\pi(\mu_{e_\beta e_\zeta})^* \xi,
  \sum_\gamma \bar e_\gamma\otimes \pi(\mu_{e_\a e_\gamma})^* \eta \rangle
\\& = \sum_\gamma \langle \pi(\mu_{e_\a e_\gamma} \mu_{e_\beta e_\gamma}^*) \xi, \eta \rangle
 = : h'_{\a\b}(x)
\end{align*}
for $x\in WAP_{cb*}(M)$.

We calculate next (all series absolutely converging) with $x,y\in WAP_{cb*}(G)$
\begin{align*}
h_{\a\b} \big( xy \big)
& = \sum_\g \langle  \pi(\mu_{e_\g e_\b}^* \mu_{e_\g e_\a}) \xi, \eta\rangle
\\& = \sum_{\g,\zeta,\lambda}  \langle (\pi_1\otimes \pi_2) (\mu_{e_\g e_\zeta}^*\otimes \mu_{e_\zeta e_\b}^*) 
 (\pi_1\otimes \pi_2) \big( \mu_{e_\g e_\lambda} \otimes \mu_{e_\lambda e_\a}) \big( \xi_1\otimes \xi_2 \big), \eta_1\otimes \eta_2\rangle
\\& = \sum_{\g,\zeta,\lambda}  \langle \pi_1 (\mu_{e_\g e_\zeta}^* \mu_{e_\g e_\lambda} ) \xi_1, \eta_1 \rangle
\langle \pi_2 ( \mu_{e_\zeta e_\b}^* \mu_{e_\lambda e_\a}) \xi_2, \eta_2\rangle
\\& = \sum_{\zeta,\lambda}  h_{\lambda\zeta} (x) 
\langle \pi_2 ( \mu_{e_\zeta e_\b}^* \mu_{e_\lambda e_\a}) \xi_2, \eta_2\rangle.
\end{align*}

For the rest of the section, we consider the algebra $M=L^\infty(G)$ where $G$ is a locally compact group. Every representation of $M_*$ is admissible. We fix still a basis $(e_\a)$ in $L^2(G)$. The usual representation of $B(G)$ on $L^2(G)$ by pointwise multiplication is unitary (has a unitary generator $W$; it is given by $WF(s,t) = F(s^{-1}t,t)$, $F\in L^2(G\times G)$). It follows that for every $x\in B(G)$
$$
\sum_\gamma (\mu_{e_\gamma e_\beta}^* \mu_{e_\gamma e_\a})(x)
 = \sum_\gamma (\mu_{e_\beta e_\gamma} \mu_{e_\a e_\gamma}^*)(x) = \delta_{\a\b} x(e). 
$$
This is considered for example in general case in \cite{haar} or can be verified directly, similarly to the formula \eqref{UUbarveta}. We write below $\e(x)=x(e)$ for $x$ in the uniform closure $\bar B(G)$ of $B(G)$ (the Eberlein algebra), so that $h_{\a\b}(x) = \delta_{\a\b}\,\e(x)$ for $x\in B(G)$.
Since the uniform norm on $WAP(G)$ is equivalent to the norm $\|\cdot\|_{cb}$, the functionals $h_{\a\b}$ are continuous with respect to the uniform norm, and the equality $h_{\a\b} = \delta_{\a\b}\,\e$ holds on $\bar B(G)$.

If $x\in B(G)$, we have also
$$
h_{\a\b} \big( xy \big) = \sum_{\zeta,\lambda} \delta_{\zeta\lambda} \e(x) \langle \pi_2 ( \mu_{e_\zeta e_\b}^* \mu_{e_\lambda e_\a}) \xi_2 \big), \eta_2\rangle = \e(x) h_{\a\b}(y).
$$
For a given $y\in WAP(G)$, pick $x\in A(G)$ such that $\e(x) \ne0$, then we have $xy \in C_0(G)\subset \bar B(G)$, and
$$
h_{\a\b}(y) = \frac{ h_{\a\b} \big( xy \big)} { \e(x)}
= \delta_{\a\b} \frac{ \e\big( xy \big)} { \e(x)} =  \delta_{\a\b} \e(y).
$$
Returning to \eqref{UUbarveta}, we get
$$
\langle U(e_\a\otimes \xi), \bar U(\bar e_\beta\otimes\eta)\rangle = \delta_{\a\b} \e(x) = \delta_{\a\b} \langle \xi,\eta \rangle
 = \langle e_\a\otimes \xi,  \bar e_\b\otimes \eta \rangle.
$$
This is valid for all $\xi\in V$, $\eta\in V^*$, and extending this equality by linearity and continuity, we see that $\langle Us,  \bar Ut\rangle = \langle s, t \rangle$ for all $s\in H\otwo V$, $t\in \bar H\otwo V^*$, so that $\bar U^*U=\Id_{H\otwo V}$.
Similarly, we show that $h'_{\a\b}=\e \delta_{\a\b}$, and
\begin{align*}
\langle \bar U^*(e_\beta\otimes  \xi), U^*(e_\a\otimes \eta)\rangle
& = \langle e_\a\otimes \xi,  e_\b\otimes \eta \rangle,
\end{align*}
which implies $U \bar U^*=\Id_{H\otwo V}$.

$U$ is thus invertible, and $U,\bar U$ being contractive, $U$ is isometric onto. This proves the following:
\begin{cor}
If $G$ is a locally compact group, then $WAP_{iso}(M) = WAP(M)$ for $M=L^\infty(G)$.
\end{cor}

There remains of course the question whether the equality $WAP_{iso}(M) = WAP(M)$ holds on a larger class of Hopf--von Neumann algebras.


\section{Appendix}

This section is intended for prepublication only. We show by detailed calculations why the axioms of the operators space do not hold for three natural choices of matrix norms on the space $B(V)$ with a general Banach space $V$.

Let $n$ be a natural number, and let $\ell_{2,n}$ denote the space $\C^n$ with the euclidean norm. The space of matrices $M_n(B(V))$ is linearly isomorphic to the space $B(\ell_{2,n}\otimes V)$ via the isomorphism $\phi_n$ given by \eqref{phin}. The resulting matrix norms on $M_n\big( B(V) \big)$ depend on the choice of the norm on $\ell_{2,n}\otimes V$.

Consider first the tensor product $\otwo$ introduced in Section 2. The first axiom of an operator space~\cite{effros-ruan}
$$
\|u\oplus v\| = \max( \|u\|, \|v\|)
$$
 is easy to check: if $u\in M_n(B(V))$, $v\in M_m(B(V))$, then 
\begin{align*}
\|u\oplus v\| &= \sup_{\sum_{k=1}^{n+m} \|\xi_k\|^2\le1} \big\| \sum_{k=1}^n e_k\otimes \sum_j u_{kj}\xi_j + \sum_{k=n+1}^{n+m} e_k\otimes \sum_j v_{kj}\xi_j \big\|
\\&= \sup_{\sum_k \|\xi_k\|^2\le1} \big( \sum_{k=1}^n \|\sum_j u_{kj}\xi_j\|^2 + \sum_{k=n+1}^{n+m} \|\sum_j v_{kj}\xi_j \big\|^2 \big)^{1/2}
\\&\le \sup_{\sum_k \|\xi_k\|^2\le1} \big( \|u\|^2 \sum_{j=1}^n\|\xi_j\|^2 + \|v\|^2 \sum_{k=n+1}^{n+m} \|\xi_j \big\|^2 \big)^{1/2}
\le \max( \|u\|, \|v\|).
\end{align*}
Choosing $\xi$ to maximize each of the two norms separately, we see that this is in fact an equality.

The second axiom reads
$$
\|\a u \b\| \le \|\a\|\,\|u\|\,\|\b\|
$$
for $\a\in M_{n,m}(\C)$, $\b\in M_{m,n}(\C)$, $u\in M_m(B(V))$.
For $\xi=\sum_j e_j\otimes\xi_j\in \ell_{2,n}\otwo V$
$$
\phi_n(\a u\b)\xi = \sum_i e_i\otimes \sum_{j,k,l} \a_{ij} u_{jk} \b_{kl} \xi_l.
$$
Suppose first that $m=n$. The map $a: \ell_{2,n}\otwo V \to \ell_{2,n}\otwo V$, $a(\xi) = \sum_i e_i\otimes\sum_j \a_{ij}\xi_j$ can alternatively be written as $a(\xi) = \sum_j \a e_j\otimes\xi_j = (\a\otimes\Id)(\sum_j e_j\otimes\xi_j)$, so we are to show that $\|(\a\otimes\Id) u (\b\otimes\Id)\| \le \|\a\|\,\|u\|\,\|\b\|$. 

The problem is that $\otwo$ is not a tensor norm, that is, in general $\|\a\otimes\Id\|\ne\|\a\|$. An easy example is provided by $n=2$ and $\a e_1=e_1-e_2$, $\a e_2=e_1+e_2$. For every $u,v\in V$ we have
\begin{align*}
\|e_1\otimes u+e_2\otimes v\|^2 &= \|u\|^2+\|v\|^2,\\
\|(\a\otimes \Id)(e_1\otimes u+e_2\otimes v)\|^2 &= \|(e_1-e_2)\otimes u+(e_1+e_2)\otimes v\|^2
\\& = \|e_1\otimes (u+v)+e_2\otimes (v-u)\|^2 = \|u+v\|^2+\|u-v\|^2.
\end{align*}
If we had $\|\a\otimes\Id\|\le\|\a\|=\sqrt2$, then we would have $\|u+v\|^2+\|u-v\|^2 \le 2 (\|u\|^2+\|v\|^2)$ for any $u,v\in V$. This holds only if $V$ is a type 2 space, which is not the case already for $\ell_p$, $p<2$. Thus, the axiom does not hold in general.

Next, we consider the Chevet-Saphar tensor products $\otimes_{d_2}$ or $\otimes_{g_2}$, and refer to \cite{ryan} for the definitions. In this case, the second axiom will hold, as both $d_2$ and $g_2$ are tensor norms.

The problem will be the first one. Let us start with the norm $g_2$ and $V=\ell_p$, $1<p<2$; we denote by $p'$ the conjugate exponent to $p$. Let $\{ \tilde e_1,\tilde e_2\}$ be the standard basis of $\ell_{2,2}$, and $\{e_n: n\in\N\}$ of $\ell_p$. Set $n=m=1$, $ue_1=e_2$, $ue_2=e_1$, $ue_n=e_n$ if $n>2$. Set also  $v=\Id$. We have $\|u\|=\|v\|=1$. Let us estimate the norm of $u\oplus v$ on $\ell_{2,2}\otimes_{g_2} \ell_p$ and show that it is greater than 1. We have
$$
(u\oplus v) (\tilde e_1\otimes e_2+ \tilde e_2\otimes e_2) = \tilde e_1\otimes e_1 + \tilde  e_2\otimes e_2.
$$
Obviously, $\|\tilde e_1\otimes e_2+\tilde e_2\otimes e_2\| = \|(\tilde e_1+\tilde e_2)\otimes e_2\| = \sqrt2$.
Let us now estimate the norm of $\tau = \tilde e_1\otimes e_1+\tilde e_2\otimes e_2$ in $E=\ell_{2,2}\otimes_{g_2} \ell_p$. This space is \cite[p.142]{ryan} isomorphic to $(\ell_{2,2}\otimes_{d_2} \ell_{p'})^*$, but also to the space of 2-summing (Hilbert-Schmidt) operators $\cal P_2 (\ell_{2,2}, \ell_p)$ from $\ell_{2,2}$ to $\ell_p$, and the operator corresponding to $\tau$ is
$$
T: y_1\tilde e_1+y_2\tilde e_2 \mapsto y_1 e_1 + y_2 e_2,
$$
that is, the identical embedding of $\ell_{2,2}$ into $\ell_p$. Its Hilbert-Schmidt norm is the least constant $C$ such that for every weakly summable sequence $(x_n)\subset \ell_{2,2}$
$$
\sum_n \|Tx_n\|_p^2 \le C^2 \sup_\phi \sum_n |\phi(x_n)|^2,
$$
where $\phi=\phi_1 \tilde e_1+\phi_2 \tilde e_2\in \ell_{2,2}$ with $\|\phi\|_2 = 1$.
Set $x_1=\tilde e_1+\tilde e_2$, $x_2=\tilde e_1-\tilde e_2$, then
$\|Tx_1\|_p^2+\|Tx_2\|_p^2 = 2^{1+2/p}$, and since in $\ell_{2,2}$ the vectors $x_1$, $x_2$ are orthogonal of norm $\sqrt2$, for every $\phi$ we have
$|\phi(x_1)|^2+|\phi(x_2)|^2 = 2$. We obtain now
$2^{1+2/p} \le 2C^2$, and $C\ge 2^{1/p}$, that is, $\|\tau\|\ge 2^{1/p}$.

Returning to $u\oplus v$, we have
$$
\|u\oplus v\| \ge \|\tau\|/\|\tilde e_1\otimes e_2+\tilde e_2\otimes e_2\| \ge 2^{1/p-1/2} >1.
$$

For the norm $d_2$, we choose $p>2$ and the same $u,v,\tau$. We have still $\|(\tilde e_1+\tilde e_2)\otimes e_2\| = \|(u\oplus v)\tau\| = \sqrt2$. The norm of $\tau$ in $\ell_{2,2}\otimes_{d_2} \ell_p$ can be estimated as
\begin{align*}
\|\tau\| &= \frac12 \| (\tilde e_1+\tilde e_2)\otimes (e_1+e_2) + (\tilde e_1-\tilde e_2)\otimes (e_1-e_2)\|
\\& \le \frac12 \sup_{\|\phi\|_{\ell_{2,2}}=1} \big( |\phi(\tilde e_1+\tilde e_2)|^2 + |\phi(\tilde e_1-\tilde e_2)|^2 \big)^{1/2} \big( \|e_1+e_2\|_p^2 + \|e_1-e_2\|_p^2)^{1/2}
\\& = \frac12 \sup_{\|\phi\|_{\ell_{2,2}}=1} \big( 2 |\phi(\tilde e_1)|^2 + 2 |\phi(\tilde e_2)|^2 \big)^{1/2} \big( 2\cdot 2^{2/p})^{1/2} 
= 2^{-1 + 1/2 +1/2 + 1/p} = 2^{1/p}.
\end{align*}

Now $\|u\oplus v\| \ge \sqrt2 / \|\tau\| \ge 2^{1/2-1/p} >1$.

\end{document}